\documentclass[leqno,11pt]{amsart}
\usepackage[colorlinks,pagebackref,hypertexnames=false]{hyperref}

\usepackage{amsmath,amsthm,amssymb,mathrsfs} 
\usepackage[alphabetic,backrefs]{amsrefs}
\usepackage{ae,aecompl}
\usepackage[margin=1in]{geometry} 
\usepackage{url}
\usepackage{comment}

\usepackage[english]{babel}

\usepackage{bookmark}

\numberwithin{equation}{section}

\renewcommand{\epsilon}{\varepsilon}

\def\<{\mathopen{}\left<}
\def\>{\right>\mathclose{}}
\def\({\mathopen{}\left(}
\def\){\right)\mathclose{}}

\newtheorem{theorem}{Theorem}

\newtheorem{proposition}[theorem]{Proposition}
\newtheorem{lemma}[theorem]{Lemma}

\theoremstyle{definition}

\newtheorem{remark}[theorem]{Remark}

\theoremstyle{definition}

\newtheorem{definition}[theorem]{Definition}

\numberwithin{theorem}{section}

\numberwithin{equation}{section}
\numberwithin{figure}{section}

\newcommand{\Comment}[2][\empty]{\ifthenelse{\equal{#1}{\empty}}{\todo[color=gray!10]{#2}}{\todo[color=gray!10,#1]{#2}}}

\author{Jason D. Lotay} 
\address[Jason D. Lotay]{Mathematical Institute, University of Oxford, U.K.}
\urladdr{\href{http://people.maths.ox.ac.uk/lotay/}{http://people.maths.ox.ac.uk/lotay/}}
\email{jason.lotay@maths.ox.ac.uk}

\author{Goncalo Oliveira} 
\address[Gon\c{c}alo Oliveira]{Department of Mathematics and CAMGSD, Instituto Superior T\'ecnico, Portugal}
\urladdr{\href{https://sites.google.com/view/goncalo-oliveira-math-webpage/home}{https://sites.google.com/view/goncalo-oliveira-math-webpage/home}}
\email{{goncalo.m.f.oliveira@tecnico.ulisboa.pt}}
\email{{galato97@gmail.com}}

\title[The mean curvature blows up at nondegenerate neck pinch singularities of LMCF]{The mean curvature blows up at nondegenerate neck pinch singularities of Lagrangian mean curvature flow} 

\date{}

\begin{document}

	\begin{abstract}
		We prove that at any  
		nondegenerate neck pinch singularity of Lagrangian mean curvature flow of surfaces the mean curvature becomes unbounded, with control on the blow-up rate. 
	\end{abstract}
	
	\maketitle
	
	
	\section{Introduction}

	\subsection{Context}
	
Lagrangian submanifolds in Calabi--Yau manifolds have been the focus of a great deal of attention in symplectic and Riemannian geometry.    Understanding how to ``decompose'' a (Hamiltonian isotopy class) of such Lagrangians in a canonical fashion, particularly using special Lagrangians, has commanded much of the research in the field and the Lagrangian mean curvature flow appeared as a natural candidate to achieve this goal \cites{Thomas,Thomas-Yau,Joyce}. The problem is also related to ideas which originated in the physics literature \cite{Douglas} and to algebraic geometry \cite{Bridgeland}.

A key aspect  of Lagrangian mean curvature flow, as in any geometric flow, is the formation of finite-time singularities, as these should relate to the intended decomposition of the initial Lagrangian.   A natural question is whether the velocity of the flow, i.e.~the mean curvature, stays bounded or not at such singularities.  For mean curvature flow of hypersurfaces in $\mathbb{R}^{n+1}$, the mean curvature must become unbounded for $n=2$ \cite{LiWang}, but can remain bounded for $n\geq 7$ \cite{Stolarski}: a central feature in the latter case is that the blow-up model is minimal (i.e.~has zero mean curvature).  

For graded Lagrangian mean curvature flow (i.e.~for the flow of Lagrangians endowed with a single-valued Lagrangian angle function), finite-time singularities are necessarily Type II \cite{Neves1} and so natural blow-up models are minimal (or even special) Lagrangians (where the Lagrangian angle is constant).  A primary example of such singularities are \emph{neck pinch singularities}, which are modelled on  Lawlor necks: these are special Lagrangian cylinders asymptotic to two transverse planes.  Neck pinch singularities  provide the basic mechanism by which a graded Lagrangian can decompose into a pair of Lagrangians. 
These particular singularities are central to Joyce's programme \cite{Joyce} for a Lagrangian mean curvature flow through singularities and with surgeries, in partial analogy to Ricci flow with surgeries (cf.~\cites{MorganTian,Perelman}), as they are expected to be a generic example of a finite-time singularity for the flow.  
	
	
	For the Lagrangian mean curvature flow of surfaces, where the ambient Calabi--Yau has real dimension 4, Székelyhidi \cite{Szeke} introduced the notion of \emph{nondegenerate} neck pinch singularities. 
	He showed they are stable under small perturbations of the initial Lagrangian, and any neck pinch singularity can be perturbed to be nondegenerate.  He also gave the first existence results of such singularities in the compact setting.
	
	While it is well known that the second fundamental form must explode at a neck pinch singularity, understanding whether the mean curvature remains bounded has until now remained an open problem. In fact, it has even been suggested that the mean curvature could remain bounded: see \cite[Chapter 7]{Wood}, and \cite{StSu} which deals with complex dimension at least $3$. As we shall see, this is not the case for nondegenerate neck pinch singularities.

	\subsection{Main result}
	
	In this article $X$ will denote a compact Calabi--Yau manifold of real dimension $4$, and we use the notion of nondegenerate neck pinch singularity introduced in \cite{Szeke}: see Section \ref{sec:Preliminaries}, and particularly Definition \ref{dfn:np}, for details. Our main result is the following.

	\begin{theorem}\label{thm:Main}
		Let $(L_t)_{t\in [0,T_0)}$ be a graded, rational, Lagrangian mean curvature flow in $X$ with uniformly bounded area ratios, whose first finite time singularity is a nondegenerate neck pinch  at $(x_0,T_0)\in X\times \mathbb{R}^+$. Then, for all $r<\tfrac{1}{2}$, the mean curvature $H_{L_t}$ satisfies
		\begin{equation}\label{eq:r.blowup}
			\lim_{t \nearrow T} (T_0-t)^r \sup_{L_t \cap B_{2 \sqrt{T_0-t}}(x_0)} |H_{L_t}| = +\infty.
		\end{equation}
	\end{theorem}
	
	\begin{remark}
		In fact, the proof will show that for any $s>0$, there is $t_s<T_0$ such that
\begin{equation}\label{eq:s.blowup}
		\sup_{L_t \cap B_{2 \sqrt{T_0-t}}(x_0)} |H_{L_t}| \gtrsim_s (T_0-t)^{s-1/2},
		\end{equation}
		for any $t_s<t<T_0$.  The significance of the exponent $1/2$ is that this is the Type I blow-up rate (as one has for self-shrinkers).
	\end{remark}
	
	\begin{remark}
	Since the rescaled flow at the singularity converges to a Lawlor neck, which is minimal but not totally geodesic, and this occurs at a slower rate, we deduce that we can compare the mean curvature $H_{L_t}$ and the second fundamental form $A_{L_t}$:
	\begin{equation}		
		\lim_{t \nearrow T}\frac{\sup_{L_t \cap B_{2 \sqrt{T_0-t}}(x_0)} |H_{L_t}|}{\sup_{L_t \cap B_{2 \sqrt{T_0-t}}(x_0)} |A_{L_t}|} =0.
		\end{equation}
	\end{remark}

	\begin{remark}
	For the neck pinch singularities of Lagrangian mean curvature flow of surfaces in \cite{LO}, it may be possible to obtain more refined information. In that case it is interesting to compare our result with the fact  that the mean curvature can remain bounded (in fact, converging to zero) at  infinite-time  singularities arising there \cite{Lee-Tsai}.
	\end{remark}

	\subsection{Organization}

	Section \ref{sec:Preliminaries} contains some preliminaries, including the definition of nondegenerate neck pinch singularity (Definition \ref{dfn:np}).  At such a singularity the (unique) tangent flow is a union of transverse planes $P_1\cup P_2$. In Section \ref{sec:Bridge} we use a method from \cite{LSS2} to construct a so-called Type I bridge (Lemma \ref{lem:Bridge}): this is a path of  length $\lesssim \sqrt{T-t_0}$ in $L_t$ connecting points $p_1(t), p_2(t)$ which when rescaled converge to the different planes $P_1,P_2$ respectively. Then, Section \ref{sec:Grading} gives a quantitative lower bound on the difference $|\theta(p_1(t))-\theta(p_2(t))|$, where $\theta$ is the Lagrangian angle: see Proposition \ref{lem:Grading}. Finally, in Section \ref{sec:Proof} we combine these two estimates with the fundamental theorem of calculus to estimate the supremum of the mean curvature $H=J\nabla \theta$, which concludes the proof of the main result.

	\subsection*{Acknowledgements}  
	JDL thanks OIST (Okinawa Institute of
Science and Technology) for hospitality during part of this project.  
	GO is partially funded by Funda\c c\~ao para a Ci\^encia e Tecnologogia (FCT) and the PRR through projects UID/04459/2025 and UID/PRR/04459/2025.

	\subsection*{Tool and computational resource disclosure}
	
	This article does not contain AI generated text. Initially, the project aimed to prove the mean curvature became unbounded at finite time singularities of the particular flows arising in \cite{LO}. AI was used for mathematical discussions at an early stage and was helpful in finding a proof strategy. After the authors finished writing a first version of the proof, AI was used to help audit technical aspects.


	\section{Preliminaries}\label{sec:Preliminaries}
	
	Let $(L_t)_{t\in[0,T_0)}$ be a graded Lagrangian mean curvature flow in $X$ with Lagrangian angle $\theta$ developing a neck pinch singularity at $(x_0,T_0)$.  We also make the finite topology assumption that $L_t$ is rational (see e.g.~\cite[Definition 7]{Szeke}) and that  $L_t$ has uniformly bounded area ratios. 	Recall from \cite[Theorem~8.2]{LSS2} that the unique tangent flow of $L_t$ at $(x_0,T_0)$ is a union of Lagrangian planes $V=P_1\cup P_2$ with the same grading, i.e.~a special Lagrangian union. 
	
\subsection{Rescaled flows}	Fixing an open coordinate neighbourhood of $x_0$, we  denote the type I scaled flow
	\begin{equation}\label{eq:TypeY}
	\widehat{L}_t=(T_0-t)^{-1/2} (L_t-x_0),
	\end{equation}
	regarded as a flow in an open set in $\mathbb{C}^2$. However, to state the definition of nondegenerate neck pinch and for later purposes, it is convenient to regard the type I scaled flow in a slightly different manner. Fix an isometric embedding of the ambient Calabi--Yau $X$ in $\mathbb R^N$ and set 
	$\tau (t)=-\log (T_0-t)$. Then, we shall consider
	\begin{equation}\label{eq:Mtau}
	M_\tau:= e^{\tau/2} (L_{T_0-e^{-\tau}} - x_0),
	\end{equation}
 for $\tau\in[-\log T_0,+\infty)$	as a flow in $\mathbb{R}^N$.  (Essentially, $M_{\tau}$ is a time-reparamaterization of $\widehat{L}_t$, with $t\mapsto \tau(t)$.)   In this perspective, the grading evolves through the drift heat equation \cite[(37)]{Szeke}:
	\begin{equation}\label{eq:grading_flow}
		\partial_\tau \theta = \Delta_{M_\tau} \theta - \frac{1}{2} x \cdot \nabla_{M_\tau} \theta,
	\end{equation}
	where $\cdot$ denotes the standard inner product in $\mathbb{R}^N$, $x$ is the position vector in $\mathbb{R}^N$, and $\Delta_{M_\tau}$, $\nabla_{M_\tau}$ the Laplacian and gradient in $M_\tau$ using the induced metric. 
	
\subsection{Nondegenerate neck pinch}	For a real-valued function $u$ on $M_\tau$, we  define its Gaussian $L^2$-norm (where $\mathcal{H}^2$ denotes 2-dimensional Hausdorff measure) as
\begin{equation}\label{eq:weighted.norm}
\|u \|_\tau^2:= \int_{M_\tau} u^2 e^{-\frac{|x|^2}{4}} d \mathcal H^2,
\end{equation}
	and its Gaussian average by
\begin{equation}\label{eq:average}
	\underline{u}_\tau:= \frac{\int_{M_\tau} u \, e^{-\frac{|x|^2}{4}} d \mathcal H^2}{\int_{M_\tau} e^{-\frac{|x|^2}{4}}d \mathcal H^2} .
	\end{equation}
	To define the notion of nondegenerate neck pinch consider a sequence of times $\tau_i \to +\infty$ and 
\begin{equation}\label{eq:theta.i}	
	\theta_i := \frac{\theta-\underline \theta_{\tau_i}}{\| \theta - \underline \theta_{\tau_i} \|_{\tau_i}},
	\end{equation}
	seen as a function on $M_{\tau+\tau_i}$ for $\tau \in (0,2]$. As argued in the discussion preceding \cite[Definition~19]{Szeke}, as $i \nearrow +\infty$, $\theta_i$ converges to a function $\theta_\infty$ on $(0,2] \times (P_1\cup P_2)$. We can then write $\theta_\infty=(\theta_{\infty,1},\theta_{\infty,2})$ with the $\theta_{\infty,j}$ functions on $(0,2] \times P_j$. The following is \cite[Definition~19]{Szeke}.

	\begin{definition}\label{dfn:np} 
		A neck pinch singularity is called \emph{nondegenerate} if:
		\begin{itemize}
			\item $(\theta_{\infty,1},\theta_{\infty,2})=(-(8\pi)^{-1/2}, (8\pi)^{-1/2})$;
			\item in a sufficiently small Darboux chart around $x_0$ and for $t$ sufficiently close to $T_0$, $L_t$ is Hamiltonian isotopic to $P_1\#P_2$, with an isotopy
			that is close to the identity near the boundary of the chart.
		\end{itemize}
		As already mentioned, any neck pinch singularity can be perturbed to become nondegenerate near the singular time \cite{Szeke}.
	\end{definition}

	\begin{remark}\label{rem:Minimizing}
		For future reference we notice the following minimizing property
\begin{equation}\label{eq:min}
\| \theta - \underline \theta_{\tau_i} \|_{\tau_i} = \inf_{b \in \mathbb{R}} \| \theta - b \|_{\tau_i}. 
\end{equation}
	\end{remark}


	\section{Type I bridge}\label{sec:Bridge}
	
	Recall that, in our setting, the type I scaling 
$\widehat{L}_t$ in \eqref{eq:TypeY}	converges smoothly away from $0$ on compact subsets of $\mathbb{C}^2$ to $V=P_1\cup P_2$. In particular, for $t$ sufficiently close to $T_0$, its intersection with $B_2\backslash B_1$ has two connected components and these are graphs over the portions of $P_1$ and $P_2$ in the annulus. Scaled back down to the annular region
\begin{equation}\label{eq:Kt}
K_t=L_t \cap B_{2\sqrt{T_0-t}} (x_0) \backslash B_{\sqrt{T_0-t}} (x_0)=K_{t,1}\sqcup K_{t,2}
\end{equation} we will refer to the components $K_{t,1}$ and $K_{t,2}$ as the $P_1$ and $P_2$ components respectively.

We now show that the distance between any two points in the $P_1,P_2$ components is uniformly bounded in terms of $\sqrt{T_0-t}$, as there is a curve  $\gamma_t$ of at most that length connecting them.   As a result, we call $\gamma_t$ a \emph{Type I bridge}.
	
	\begin{lemma}\label{lem:Bridge}
		There is 
		$C>0$ such that, for all $t$ sufficiently near $T_0$ and for any points $p_1\in K_{t,1}$ and $p_2\in K_{t,2}$ 
		as in \eqref{eq:Kt}, 
		 there is a curve
		\begin{equation}\gamma_t \subset L_t \cap B_{2\sqrt{T_0-t}} (x_0)
		\end{equation}
connecting $p_1$ and $p_2$,		with
\begin{equation}		
		\mathrm{Length}(\gamma_t) \leq C \sqrt{T_0-t}.
		\end{equation}
	\end{lemma}
	\begin{proof}  The argument is essentially already given within the proof of \cite[Lemma 8.1]{LSS2}, but we provide the details.
	
	Consider the type I scaled flow $\widehat{L}_{t}$ in \eqref{eq:TypeY} (and recall the relation between the rescaled flow $M_{\tau}$ in \eqref{eq:Mtau} and $\widehat{L}_t$ so as to compare with the statements in \cite{LSS2}).  For  $t$  sufficiently near $T_0$ it follows from \cite[Lemma 3.5 (2)]{LSS2} that $\widehat{L}_{t}\cap B_3$ is \emph{almost calibrated}, i.e.~the variation of the Lagrangian angle is less than $\pi-\delta$ for some $\delta>0$.   Hence, we have by \cite[Lemma 7.2]{Neves1} a uniform lower bound for the area of intrinsic unit balls in $\widehat{L}_t$ centred at a point $x\in B_2\cap\widehat{L}_t$:
	\begin{equation}\label{eq:area.1}
	\mathcal{H}^2(\widehat{B}_{\widehat{L}_t}(x,1))>C_1.
	\end{equation}
	We also have a uniform upper bound for the area of $B_2\cap \widehat{L}_t$ for all $t$ sufficiently near $T_0$:
\begin{equation}\label{eq:area.2}
\mathcal{H}^2(B_2\cap\widehat{L}_t)<C_2,
\end{equation}	
	 since we assumed we had uniformly bounded area ratios for the original flow $L_t$.
	
	There are no compact almost calibrated Lagrangians in $\mathbb{C}^2$, so $\widehat{L}_{t}\cap B_2$ has either 1 or 2 connected components.  If $\widehat{L}_t\cap B_2$ has 2 connected components for some $t$ sufficiently close to $T_0$, then it follows from \cite[Section 6]{Neves1} that there is no singularity at $(x_0,T_0)$, which is a contradiction.  Hence, we must have that $\widehat{L}_t\cap B_2$ is connected for all $t$ sufficiently close to $T_0$.  
	Using this connectedness together with the area bounds \eqref{eq:area.1}--\eqref{eq:area.2} implies that there exists $C>0$ so that any point $p\in \widehat{L}_t\cap B_2$ can be connected to $x_0$ with a curve of length at most $C/2$.  
	
	Rescaling back to the original flow and choosing $p_1,p_2$ in the $P_1,P_2$ components $K_{t,1},K_{t,2}$ gives the claimed result.
	%
	%
%
	\end{proof}


	\section{Grading difference estimate}\label{sec:Grading}
	
	The goal of this section is to prove the following result, which gives a lower bound for  the variation in the Lagrangian angle $\theta$ between certain points on the $P_1,P_2$ components in terms of the remaining time to the singular time $T_0$.
	
	\begin{proposition}\label{lem:Grading}
For all $s>0$ there is $t_s<T_0$ such that for all $t \in (t_s , T_0)$ the following holds. There are points $p_1(t)$, $p_2(t)$  in the $P_1,P_2$ components respectively of 
$K_t$ in \eqref{eq:Kt}, such that
\begin{equation}\label{eq:theta.bound}
		 |\theta(p_2(t)) - \theta(p_1(t))| \gtrsim_s (T_0-t)^{s}.
		 \end{equation}
	\end{proposition}
	\begin{proof}
		Let $s>0$. We shall prove the result using the rescaled flow $M_\tau$ in \eqref{eq:Mtau}. In this setting, the claim is equivalent to finding $\tau_s>0$ and $\tilde p_1(\tau)$, $\tilde p_2 (\tau)$, respectively in the $P_1$ and $P_2$ components of $M_\tau \cap B_2\backslash B_1$ (using the same notion as introduced at the start of Section \ref{sec:Bridge}), such that
		\begin{equation}\label{eq:variation inequality tau}
			|\theta(\tilde p_2(\tau)) - \theta(\tilde p_1(\tau))| \gtrsim_s e^{-s\tau}.
		\end{equation} 
Consider a sequence of times $\tau_i\to+\infty$ so that $\theta_i$ in \eqref{eq:theta.i} has a limit 
\begin{equation}\label{eq:theta.infty}
\theta_{\infty}=(-(8\pi)^{-1/2},(8\pi)^{-1/2})
\end{equation}
 on $V=P_1\cup P_2$ as $i\to\infty$, as in Definition \ref{dfn:np}.  

We begin with the following, which gives a first control on the rate of change of $\|\theta-\underline{\theta}_{\tau_i}\|_{\tau_i}$.

\begin{lemma}\label{lem:slow.decay}
For all $q>0$ there are infinitely many $\tau_i$ such that 
\begin{equation}\label{eq:slow.decay}
\|\theta-\underline{\theta}_{\tau_{i}+1}\|_{\tau_{i}+1}\geq e^{-q}\|\theta-\underline{\theta}_{\tau_i}\|_{\tau_i}.
\end{equation} 
\end{lemma}

\begin{proof}
Recalling \eqref{eq:weighted.norm} and \eqref{eq:average}, we estimate the ratio		
		\begin{align}
			\frac{\| \theta - \underline \theta_{\tau_{i}+1} \|_{\tau_{i}+1}}{\| \theta - \underline \theta_{\tau_i} \|_{\tau_i}} & = 
			\Big\Vert \frac{ \theta - \underline \theta_{\tau_i}}{\| \theta - \underline \theta_{\tau_i} \|_{\tau_i}}  
			- 		\frac{\underline \theta_{\tau_{i}+1} - \underline \theta_{\tau_i} }{\| \theta - \underline \theta_{\tau_i} \|_{\tau_i}} \Big\Vert_{\tau_{i}+1} \nonumber\\
			& = 
			\Big\Vert \theta_i
			- 
			\frac{\underline \theta_{\tau_{i}+1} - \underline \theta_{\tau_i} }{\| \theta - \underline \theta_{\tau_i} \|_{\tau_i}} \Big\Vert_{\tau_{i}+1} \nonumber\\
			& \geq \inf_{b \in \mathbb{R}} \| \theta_i -b \|_{\tau_{i}+1},\label{eq:min.ineq}
		\end{align}
		where in \eqref{eq:min.ineq} we used the minimizing property \eqref{eq:min} in Remark \ref{rem:Minimizing}. Furthermore, notice that as $i \to +\infty$ we have $M_{\tau_{i+1}} \to P_1 \cup P_2$ and $\theta_i\to\theta_\infty$ which takes opposite values in these (equal Gaussian area) planes as in \eqref{eq:theta.infty}. Hence, the normalization in \eqref{eq:theta.infty} and the inequality \eqref{eq:min.ineq} imply that
\begin{equation}		
		\liminf_{i \to +\infty} \frac{\| \theta - \underline \theta_{\tau_{i}+1} \|_{\tau_{i}+1}}{\| \theta - \underline \theta_{\tau_i} \|_{\tau_i}}  \geq 1. 
		\end{equation}
The estimate \eqref{eq:slow.decay} follows.
\end{proof}		
		

We may now propagate the decay estimate \eqref{eq:slow.decay} using a three annulus lemma for $\theta-\underline{\theta}_{\tau_i}$, taken from \cite{Szeke}.	
	
\begin{lemma}\label{lemma:iterative.slow.decay}
	Let $0<q<1/2$. Then, there are infinitely many $\tau_i>0$ such that for all $k \in \mathbb N$
	\begin{equation}\label{eq:Iterative}
		\| \theta - \underline \theta_{\tau_{i}+k+1} \|_{\tau_{i}+k+1} \geq e^{-q} \| \theta - \underline \theta_{\tau_{i}+k} \|_{\tau_{i}+k} \geq \ldots \geq e^{-q(k+1)} \| \theta - \underline \theta_{\tau_i} \|_{\tau_i} . 
	\end{equation}
\end{lemma}
\begin{proof}
		Given any $0<q<1/2$ we can apply Lemma \ref{lem:slow.decay} to deduce the existence of infinitely many $\tau_i$ such that \eqref{eq:slow.decay} holds.  Since $ -2q \in \mathbb R \backslash \mathbb Z$, the conditions of \cite[Lemma~20]{Szeke} are met, which can then be applied iteratively $k$ times to deduce \eqref{eq:Iterative}.
\end{proof}

The next result is the key to the proof of Proposition \ref{lem:Grading}. It roughly says that given sufficient growth of the deviation of $\theta$ from its average, we can obtain a definite amount of variation of the Lagrangian angle as we pass between the $P_1$ and $P_2$ components of the rescaled flow. 

\begin{lemma}\label{lem:intermediate.grading}
	Let $0<\eta< \tfrac{1}{\sqrt{8\pi}}$. There are $q_0>0$ and $\tau_\ast$ such that for all $\tilde \tau>\tau_\ast$ and $0<q<q_0$, if
	\begin{equation}\label{eq:slow_decay_2_steps}
	\begin{aligned}
		\| \theta - \underline \theta_{\tilde \tau+2} \|_{\tilde \tau+2} & \geq e^{-q} \| \theta - \underline \theta_{\tilde \tau+1} \|_{\tilde \tau+1}, \\
		\| \theta - \underline \theta_{\tilde \tau+1} \|_{\tilde \tau+1} & \geq e^{-q} \| \theta - \underline \theta_{\tilde \tau} \|_{\tilde \tau},
	\end{aligned}
	\end{equation}
	then for any $\tau' \in[1,2]$ there are $\tilde p_1 (\tilde \tau+\tau')$ and $\tilde p_2(\tilde \tau+\tau')$ on the $P_1$ and $P_2$ components of $M_{\tilde \tau+\tau'} \cap B_2\backslash B_1$ respectively  such that 
	\begin{equation}\label{eq:separation_proof}
		|\theta(\tilde p_1(\tilde \tau+\tau')) - \theta(\tilde p_2(\tilde \tau +\tau'))| \geq \eta \| \theta - \underline \theta_{\tilde \tau+1} \|_{\tilde \tau+1}.
	\end{equation}
\end{lemma}
\begin{proof}
	We argue by contradiction.  To this end, suppose that for all $q_0>0$ and $\tau_\ast$, there are $0<q<q_0$ and $\tilde \tau >\tau_\ast$ such that \eqref{eq:slow_decay_2_steps} holds but \eqref{eq:separation_proof} fails. This implies that there is a sequence $q_i \searrow 0$ and $\tilde \tau_i \nearrow +\infty $ such that \eqref{eq:slow_decay_2_steps} holds with $q_i$ instead of $q$, but there exists $\tau_i'\in [1,2]$ such that
	\begin{equation}\label{eq:separation_negation}
		|\theta(\tilde p_1) - \theta(\tilde p_2)| < \eta \| \theta - \underline \theta_{\tilde\tau_i+1} \|_{\tilde\tau_i+1}
	\end{equation}
	holds for all $\tilde{p}_1,\tilde{p}_2$ on the $P_1,P_2$ components of $M_{\tilde{\tau}_i+\tau_i'}\cap B_2\setminus B_1$.
	
	At this point it is convenient to consider 
\begin{equation}	\label{eq:tilde.theta.i}
	\tilde \theta_i (x,\tau'):= \frac{\theta (x,\tilde \tau_i+\tau') - \underline \theta_{\tilde\tau_i}}{\| \theta - \underline \theta_{\tilde\tau_i+1} \|_{\tilde\tau_i+1}} ,
	\end{equation}
	defined for $\tau' \in [0,2]$, and its average
\begin{equation}	\label{eq:tilde.average}
	\underline{\tilde \theta}_{i,\tau'} = \frac{\int_{M_{\tilde\tau_i+\tau'}} \tilde \theta_i(x,\tau') e^{-|x|^2/4} d \mathcal H^2 }{\int_{M_{\tilde\tau_i+\tau'}}e^{-|x|^2/4} d \mathcal H^2}=\frac{\underline{\theta}_{\tilde{\tau}_i+\tau'}-\underline{\theta}_{\tilde{\tau}_i}}{\| \theta - \underline \theta_{\tilde\tau_i+1} \|_{\tilde\tau_i+1}},
	\end{equation}
in a similar way to \eqref{eq:average}.  
	Then, \eqref{eq:separation_negation} becomes
	\begin{equation}\label{eq:separation_negation_2}
		|\tilde \theta_i(\tilde p_1, \tau_i') - \tilde \theta_i(\tilde p_2,\tau_i')| < \eta
	\end{equation}
	 for all $\tilde{p}_1,\tilde{p}_2$ on the $P_1,P_2$ components of $M_{\tilde{\tau}_i+\tau_i'}\cap B_2\setminus B_1$.  Our goal is to derive a contradiction to \eqref{eq:separation_negation_2} by sending $i\to\infty$ and analysing limiting quantities on the planes $V=P_1\cup P_2$.

	To obtain this limit as $i\to\infty$, we first note the following basic but key estimates.
	 
	 \begin{lemma}\label{lem:estimates_normalized}
	 The quantities $\tilde{\theta}_i$ in \eqref{eq:tilde.theta.i} and $\underline{\tilde{\theta}}_{i,\tau'}$ in \eqref{eq:tilde.average} satisfy
	 	\begin{equation}\label{eq:estimates_normalized}
		\begin{aligned}
			\| \tilde \theta_i (\cdot , 0) \|_{\tilde\tau_i} & \leq e^{q_i}, \\
			\| \tilde \theta_i - \underline{\tilde \theta}_{i,1} \|_{\tilde\tau_i+1} & = 1, \\
			\| \tilde \theta_i - \underline{\tilde \theta}_{i,2} \|_{\tilde\tau_i+2} & \geq e^{-q_i}.
		\end{aligned}
	\end{equation}
	 \end{lemma}
	
	\begin{proof} This follows from a short computation using \eqref{eq:slow_decay_2_steps} (with $q_i$ in place of $q$) together with \eqref{eq:tilde.theta.i}--\eqref{eq:tilde.average}.
	\end{proof}

	We now observe, using Kato's inequality and the fact that $\theta$ satisfies \eqref{eq:grading_flow}, that $|\tilde \theta_i|$ is a subsolution to the drift heat equation \eqref{eq:grading_flow} in $\tau' \in [0,2]$.  This enables us to deduce the following pointwise bounds for $\tilde{\theta}_i$ given the integral bounds in \eqref{eq:estimates_normalized}.
	
\begin{lemma}\label{lem:delta.bounds}
For all $\delta\in (0,1)$ there exist constants $p>1$ and $C>0$, depending on $\delta$, so that for all $i$ and all $\tau'\in [\delta,2]$ we have
\begin{equation}\label{eq:delta.bound.0}
|\tilde{\theta}_i(x,\tau')|^2\leq Ce^{\frac{|x|^2}{4p}}.
\end{equation}	
\end{lemma}

\begin{proof}
Given that $|\tilde{\theta}_i|$ is a subsolution of the drift heat equation we may apply \cite[Proposition 18(b)]{Szeke} to $|\tilde\theta_i|$, using \eqref{eq:estimates_normalized}.  Hence, for all $0<\delta<1$ we obtain $p>1$, such that for $\tau' \in [\delta,1]$ we have a $C^0$ estimate of the form
 \begin{equation}\label{eq:delta.bound}
 |\tilde \theta_i (x,\tau')|^2 \lesssim_\delta e^{2q_i} e^{\frac{|x|^2}{4p}} \lesssim e^{\frac{|x|^2}{4p}},
 \end{equation}
 where we used that $q_i\searrow 0$. 
 
	We now show that a similar estimate to \eqref{eq:delta.bound} holds for $\tau' \in [1,2]$. Indeed, if we let $s \in [0,1]$ and $p(s)=1+e^s$, then \cite[Proposition 18(a)]{Szeke} gives
	\begin{equation}\label{eq:Prop.18a.Szeke}
		\left( \int_{M_{\tilde \tau_i +s}} |\tilde \theta_i (x, s)|^{p(s)} e^{-\frac{|x|^2}{4}} d \mathcal{H}^2 \right)^{\frac{1}{p(s)}} \lesssim \|\tilde \theta_i (\cdot , 0 ) \|_{\tilde \tau_i} \lesssim 1.
	\end{equation}
	As $p(s)>2$, we can combine \eqref{eq:Prop.18a.Szeke} with  H\"older's inequality and the uniform Gaussian area bound to obtain
\begin{equation}	
	 \|\tilde \theta_i (\cdot , s ) \|_{\tilde \tau_i+s} \lesssim 1.
	 \end{equation}
	Then, given any $\tau' \in [1,2]$ we can again apply \cite[Proposition 18(b)]{Szeke} to the interval $[\tau'-1,\tau']$ instead of $[0,1]$ (with $\delta=1/2$ say) which gives
\begin{equation}\label{eq:delta.bound.2}
	|\tilde \theta_i (x,\tau')|^2 \lesssim  e^{\frac{|x|^2}{4p}}, 
\end{equation}
	for some $p>1$. This estimate is independent of $\tau'$ so it holds uniformly for all $\tau' \in [1,2]$. Furthermore, the bound \eqref{eq:delta.bound.2} is also uniform in $i$ and so the same is true for \eqref{eq:delta.bound.2}.  
	
	Combining \eqref{eq:delta.bound} and \eqref{eq:delta.bound.2} yields the result.
	\end{proof}
	
Recall that, as $\tau \to +\infty$, $M_\tau$ smoothly converges to $V=P_1\cup P_2$ away from the origin, and the drift heat equation is uniformly parabolic.   We deduce from Lemma \ref{lem:delta.bounds} that, as $i\to\infty$, $\tilde \theta_i$ converges to $\tilde \theta_\infty=(\tilde \theta_{\infty,1},\tilde \theta_{\infty,2})$ smoothly on compact subsets of $(0,2] \times (V \backslash \{0\})$. Furthermore, 
$\tilde{\theta}_{\infty}$ satisfies the drift heat equation \eqref{eq:grading_flow}	on $(0,2] \times (V \backslash \{0\})$.  We have thus obtained the desired limiting quantity $\tilde{\theta}_{\infty}$ on $V$ and we may write
	\begin{equation}
	\|\tilde{\theta}_{\infty}\|^2_{\tau'}=\int_V|\tilde{\theta}(x,\tau')|^2e^{-\frac{|x|^2}{4}} d\mathcal{H}^2
	\end{equation}
	and let $\underline{\tilde{\theta}}_{\infty,\tau'}$ denote the average of $\tilde{\theta}_{\infty}$ at time $\tau'$ in an analogous way to \eqref{eq:tilde.average}.
	
	To obtain our contradiction to \eqref{eq:separation_negation_2} we need to demonstrate that we have suitable norm convergence of $\tilde{\theta}_i$ to $\tilde{\theta}_{\infty}$, so that we can send $i\to\infty$ in the estimates in Lemma \ref{lem:estimates_normalized}.  We first show that we have convergence of the averages.

	\begin{lemma}\label{lem:average.convergence} For all $\tau'\in (0,2]$, we have 
	\begin{equation}\label{eq:average.convergence}
		\lim_{i\to\infty}\underline{\tilde{\theta}}_{i,\tau'}=\underline{\tilde{\theta}}_{\infty,\tau'}.
	\end{equation} 
\end{lemma}

\begin{proof}
We first notice by the convergence of $M_\tau$ to $V$ as $\tau\to\infty$ that
\begin{equation}
\lim_{\tau\to\infty}\int_{M_{\tau}}e^{-\frac{|x|^2}{4}}d\mathcal{H}^2=\int_Ve^{-\frac{|x|^2}{4}}d\mathcal{H}^2=8\pi.
\end{equation}
Hence the denominators in the definition of the averages converge, so we need only consider the numerators.

Given any $R>r>0$ we denote the closed annulus with radii $r,R$ by
	\begin{equation}\label{eq:annulus}
	A_{r,R}=\overline{B}_R(0)\setminus B_{r}(0).
	\end{equation}
	Smooth convergence on compact subsets away from the origin  of  $\tilde{\theta}_i$ to $\tilde{\theta}_{\infty}$ as $i\to\infty$ gives
	\begin{equation}\label{eq:convergence.annulus.average}
	\lim_{i\to\infty}\int_{M_{\tilde{\tau}_i+\tau'}\cap A_{r,R}} \tilde\theta_i(x,\tau') e^{-\frac{|x|^2}{4}}d\mathcal{H}^2=\int_{V\cap A_{r,R}} \tilde{\theta}_{\infty}(x,\tau') e^{-\frac{|x|^2}{4}}d\mathcal{H}^2.
	\end{equation}
We therefore need only show that there is no loss at the origin or at infinity as $i\to\infty$ to obtain \eqref{eq:average.convergence}.

Near the origin, say in $M_{\tilde{\tau}_i+\tau'}\cap B_{r}(0)$, we have from the $C^0$ estimate \eqref{eq:delta.bound.0} that $|\tilde \theta_i (x,\tau')| \lesssim  1$, and thus (uniformly in $i$)
\begin{equation}	
	\int_{M_{\tilde{\tau}_i+\tau'}\cap B_{r}(0)} |\tilde\theta_i| e^{-\frac{|x|^2}{4}}d\mathcal{H}^2 \lesssim \mathcal{H}^2 ( M_{\tilde{\tau}_i+\tau'}\cap B_{r}(0) ) \lesssim r^2, 
\end{equation}
	by the uniform area ratio bound.  Hence,
\begin{equation}	\label{eq:no.loss.zero.average}
	\lim_{r \to 0} \sup_i \int_{M_{\tilde{\tau}_i+\tau'}\cap B_{r}(0)} |\tilde\theta_i| e^{-\frac{|x|^2}{4}}d\mathcal{H}^2 =0.
	\end{equation}

To study what happens at infinity, if we let $p= 1+ e^{\tau'}>2$, since $|\tilde{\theta}_i|$ is a subsolution of the drift heat equation, we may combine  \cite[Proposition 18(b)]{Szeke} and \eqref{eq:estimates_normalized} to obtain
\begin{equation}	\label{eq:p.tilde.bound.infinity}
	\left( \int_{M_{\tilde \tau_i +\tau'}} |\tilde \theta_i (x, s)|^{p} e^{-|x|^2/4} d \mathcal{H}^2 \right)^{\frac{1}{p}} \lesssim \|\tilde \theta_i (\cdot , 0) \|_{\tilde \tau_i} \lesssim 1.
	\end{equation}
	As $p>2$ we can apply H\"older's inequality to obtain
	\begin{align}
		\int_{M_{\tilde{\tau}_i+\tau'} \backslash B_R(0)} |\tilde\theta_i| e^{-\frac{|x|^2}{4}}d\mathcal{H}^2 & \lesssim \left( \int_{M_{\tilde{\tau}_i+\tau'}\backslash B_R(0)} |\tilde\theta_i|^{p} e^{-\frac{|x|^2}{4}}d\mathcal{H}^2 \right)^{\frac{1}{p}}  \left(  \int_{M_{\tilde{\tau}_i+\tau'} \backslash B_R(0)} e^{-\frac{|x|^2}{4}}d\mathcal{H}^2  \right)^{1-\frac{1}{p}}\nonumber\\ 
		&\lesssim e^{-\frac{R^2}{4}(1-\frac{1}{p})},\label{eq:no.loss.infinity.average.0}
	\end{align}
	by using the uniform area bound and \eqref{eq:p.tilde.bound.infinity}.  Since the estimate \eqref{eq:no.loss.infinity.average.0} is uniform in $i$, we have that
	\begin{equation}\label{eq:no.loss.infinity.average}
	\lim_{R\to\infty}\sup_i\int_{M_{\tilde{\tau}_i+\tau'}\setminus B_R(0)} |\tilde{\theta}_i|e^{-\frac{|x|^2}{4}}d\mathcal{H}^2=0.
\end{equation}	 Combining \eqref{eq:convergence.annulus.average}, \eqref{eq:no.loss.zero.average} and \eqref{eq:no.loss.infinity.average} yields the result.
	\end{proof}

Given the convergence of averages, we now show that we have convergence of $L^2$ norms.
	
\begin{lemma}\label{lem:norm.convergence}
For all $\tau'\in(0,2]$ we have
\begin{equation}\label{eq:norm.convergence.1}
\lim_{i\to\infty} \| \tilde \theta_{i}-\underline{\tilde{\theta}}_{i,\tau'} \|_{\tilde \tau_i+\tau'} = \| \tilde{\theta}_{\infty}-\underline{\tilde{\theta}}_{\infty,\tau'} \|_{\tau'}.
\end{equation}
\end{lemma} 
	
	\begin{proof}
	Smooth convergence on compact subsets away from the origin  of  $\tilde{\theta}_i$ to $\tilde{\theta}_{\infty}$ as $i\to\infty$ and convergence of the averages by Lemma \ref{lem:average.convergence} gives convergence on any annulus as in \eqref{eq:annulus}: 
	\begin{equation}\label{eq:convergence.annulus}
	\lim_{i\to\infty}\int_{M_{\tilde{\tau}_i+\tau'}\cap A_{r,R}} |\tilde\theta_i(x,\tau')-\underline{\tilde{\theta}}_{i,\tau'}|^2 e^{-\frac{|x|^2}{4}}d\mathcal{H}^2=\int_{V\cap A_{r,R}} |\tilde{\theta}_{\infty}(x,\tau')-\underline{\tilde{\theta}}_{\infty,\tau'}|^2 e^{-\frac{|x|^2}{4}}d\mathcal{H}^2.
	\end{equation}
To obtain \eqref{eq:norm.convergence.1}	we therefore must (as before) exclude norm being lost at the origin and at infinity.

At the origin, we again work in $M_{\tilde{\tau}_i+\tau'}\cap B_{r}(0)$. From the $C^0$ estimate \eqref{eq:delta.bound.0} we have that $|\tilde \theta_i (x,\tau')| \lesssim  1$  (uniformly in $i$) and hence, by \eqref{eq:average.convergence}, we deduce that
$|\tilde{\theta}_i(x,\tau')-\underline{\tilde{\theta}}_{i,\tau'}|\lesssim 1$.  Therefore, we can argue using the uniform area ratio bounds that we have
\begin{equation}	\label{eq:no.loss.zero}
	\lim_{r \to 0} \sup_i \int_{M_{\tilde{\tau}_i+\tau'}\cap B_{r}(0)} |\tilde\theta_i-\underline{\tilde{\theta}}_{i,\tau;}|^2 e^{-\frac{|x|^2}{4}}d\mathcal{H}^2 =0.
	\end{equation}

	As for estimating what happens at infinity, we start again with the $C^0$ estimate \eqref{eq:delta.bound.0}, which gives $|\tilde \theta_i (x,\tau')|^2 \lesssim  e^{|x|^2/(4p)}$ for some $p>1$ and thus
	\begin{align}
		\int_{M_{\tilde{\tau}_i+\tau'}\setminus B_R(0) } |\tilde\theta_i|^2 e^{-\frac{|x|^2}{4}}d\mathcal{H}^2 & = \sum_{m=0}^\infty \int_{M_{\tilde{\tau}_i+\tau'}\cap A_{R+m,R+m+1} } |\tilde\theta_i|^2 e^{-\frac{|x|^2}{4}}d\mathcal{H}^2 \nonumber\\
		& \lesssim \sum_{m=0}^\infty  \mathcal{H}^2 \left( M_{\tilde{\tau}_i+\tau'}\cap A_{R+m,R+m+1}  \right)  \sup_{x \in A_{R+m,R+m+1}} e^{- (\frac{1}{4} -\frac{1}{4p} ) |x|^2 } \nonumber\\
		& \lesssim \sum_{m=0}^\infty  (R+m+1)^2  e^{- (\frac{1}{4} -\frac{1}{4p} ) |m+R|^2 },\label{eq:no.loss.infinity}
	\end{align}
	where in \eqref{eq:no.loss.infinity} we used again the uniform control on the area ratio bound. Now, the estimate on the right hand side of \eqref{eq:no.loss.infinity} is uniformly bounded in $i$ and converges to zero as $R \to +\infty$.  Using Lemma \ref{lem:average.convergence}, uniform area ratio bounds and \eqref{eq:no.loss.infinity}, we deduce that
\begin{align}
\sup_i \int_{M_{\tilde{\tau}_i+\tau'}\setminus B_R(0) } &|\tilde\theta_i-\underline{\tilde{\theta}}_{i,\tau'}|^2  e^{-\frac{|x|^2}{4}}d\mathcal{H}^2\nonumber\\
&\leq \sup_i \int_{M_{\tilde{\tau}_i+\tau'}\setminus B_R(0) } 2|\tilde\theta_i|^2 e^{-\frac{|x|^2}{4}}d\mathcal{H}^2
+\sup_i |\underline{\tilde{\theta}}_{i,\tau'}|^2\int_{M_{\tilde{\tau}_i+\tau'}\setminus B_R(0) } 2e^{-\frac{|x|^2}{4}}d\mathcal{H}^2\to 0
\label{eq:no.loss.infinity.2}
\end{align}	
as $R\to\infty$.

	 Combining \eqref{eq:no.loss.infinity.2} with \eqref{eq:convergence.annulus} and \eqref{eq:no.loss.zero} completes the proof.
	\end{proof}

Lemma \ref{lem:norm.convergence} allows us to send $i\to\infty$ in the last two estimates in \eqref{eq:estimates_normalized} as desired, which gives
	\begin{equation}\label{eq:difference_from_average_normalized}
		\begin{aligned}
			\| \tilde \theta_\infty - \underline{\tilde \theta}_{\infty,1} \|_{1} & = 1 \\
			\| \tilde \theta_\infty - \underline{\tilde \theta}_{\infty,2} \|_{2} & \geq 1,
		\end{aligned}
	\end{equation}
	recalling that $q_i\to 0$ as $i\to\infty$.  (Note here that the subscripts $1$ and $2$ refer to  $\tau'=1$ and $2$.)  
	
We now use the fact that $\tilde{\theta}_{\infty}$ satisfies the drift heat equation to derive the following.

\begin{lemma}\label{lem:constancy}
We have that $\tilde{\theta}_{\infty}(x,\tau')$ is constant in $x\in V$ for all $\tau'\in[1,2]$ and  $\tilde{\theta}_{\infty}(x,\tau')-\underline{\tilde\theta}_{\infty,\tau'}$ is constant in $x\in V=P_1\cup P_2$ and $\tau'\in [1,2]$.
\end{lemma}	

\begin{proof}
We regard both $\tilde \theta_\infty$ and $\underline{\tilde \theta}_{\infty,\tau'}$ as functions of $\tau'$ and compute
	\begin{align}
		\partial_{\tau'} \| \tilde \theta_\infty - \underline{\tilde \theta}_{\infty,\tau'}\|^2_{V} & = 2\int_{V}  (\tilde \theta_\infty -\underline{\tilde \theta}_{\infty,\tau'}) \partial_{\tau'} (\tilde \theta_\infty - \underline{\tilde \theta}_{\infty,\tau'}) e^{-\frac{|x|^2}{4}} d \mathcal H^2 \nonumber\\
		& = 2\int_{V}  (\tilde \theta_\infty - \underline{\tilde \theta}_{\infty,\tau'}) \partial_{\tau'} \tilde \theta_\infty  e^{-\frac{|x|^2}{4}} d \mathcal H^2 
		-
		2 (\partial_{\tau'} \underline{\tilde \theta}_{\infty,\tau'}) \int_{V}  (\tilde \theta_\infty - \underline{\tilde \theta}_{\infty,\tau'})  e^{-\frac{|x|^2}{4}} d \mathcal H^2 \nonumber\\
		& = 2\int_{V}  (\tilde \theta_\infty - \underline{\tilde \theta_\infty}_{\tau'}) \partial_{\tau'} \tilde \theta_\infty  e^{-\frac{|x|^2}{4}} d \mathcal H^2 ,
	\end{align}
	where we have used that $ \underline{\tilde \theta}_{\infty,\tau'}$ is constant in space and that $\tilde \theta_\infty - \underline{\tilde \theta}_{\infty,\tau'}$ has zero Gaussian average by definition of $\underline{\tilde\theta}_{\infty,\tau'}$. We may then use the drift heat equation \eqref{eq:grading_flow} and integration by parts to obtain
	\begin{equation}\label{eq:evolution}
		\partial_{\tau'} \| \tilde \theta_\infty - \underline{\tilde \theta}_{\infty,\tau'} \|_{V}^2  = -2\int_{V} \nabla (\tilde \theta_\infty - \underline{\tilde \theta}_{\infty,\tau'} ) \cdot \nabla \tilde \theta_\infty  e^{-\frac{|x|^2}{4}} d \mathcal H^2  = -2 \int_{V} |\nabla \tilde \theta_\infty|^2 e^{-\frac{|x|^2}{4}} d \mathcal H^2,
	\end{equation}
	which is non-positive. Combining with \eqref{eq:difference_from_average_normalized} implies that $\tilde \theta_\infty - \underline{\tilde \theta}_{\infty,\tau'}$ is constant in $\tau'\in [1,2]$.  It then follows from \eqref{eq:evolution} that $\nabla\tilde{\theta}_{\infty}=0$ for $\tau'\in[1,2]$, which yields the claimed result.
\end{proof}
	
\textit{We now complete the proof of Lemma \ref{lem:intermediate.grading}.}	
	
Since $P_1$ and $P_2$ both have Gaussian area $4\pi$, by Lemma \ref{lem:constancy} we see that on $V=P_1\cup P_2$ we have
\begin{equation}
\tilde{\theta}_{\infty}-\underline{\tilde{\theta}}_{\infty,\tau'}=\frac{1}{2}(\tilde{\theta}_{\infty,1}(\tau')-\tilde{\theta}_{\infty,2}(\tau'),\tilde{\theta}_{\infty,2}(\tau')-\tilde{\theta}_{\infty,1}(\tau'))
\end{equation}
where the subscripts on the right-hand side denote the values on the planes $P_1$ and $P_2$. Hence, by \eqref{eq:difference_from_average_normalized}, we have that
\begin{equation}
2\pi (\tilde{\theta}_{\infty,1}-\tilde{\theta}_{\infty,2})^2=1
\end{equation}
for all $\tau'\in[1,2]$. Hence, for all $\tau'\in[1,2]$, $\tilde{p}_1\in P_1$ and $\tilde{p}_2\in P_2$ we have
\begin{equation}
|\tilde \theta_{\infty}(\tilde{p}_1,\tau') - \tilde \theta_{\infty}(\tilde{p}_2,\tau')| = \frac{1}{\sqrt{2\pi}}>2\eta,
\end{equation}
which then 
	contradicts \eqref{eq:separation_negation_2} as desired. 
\end{proof}

\emph{We now finalize the proof of Proposition \ref{lem:Grading}.}

	Let $s>0$ and $0<\eta<1/\sqrt{8\pi}$.  Let $q_0>0$ and $\tau_*$ be given by Lemma \ref{lem:intermediate.grading} and  choose $0<q<\min \{ s,q_0,1/2\}$. By Lemma \ref{lemma:iterative.slow.decay}, since $0<q<1/2$ we may choose $\tau_i>\tau_*$  so that \eqref{eq:Iterative} holds.  Note that any $\tau\geq\tau_i+1$ can be written as 
	\begin{equation}\label{eq:tau.formula}
	\tau = \tau_i +k+\tau' 
	\end{equation}
	for some $k\in\mathbb{N}$ and $\tau'\in [1,2]$.  
	 
By Lemma \ref{lemma:iterative.slow.decay}, for all $k \in \mathbb{N}$ we have that \eqref{eq:slow_decay_2_steps} in  Lemma \ref{lem:intermediate.grading} holds with $q$ as chosen (since $0<q<q_0$) and $\tilde \tau=\tau_i+k$.    
	 Then, it follows from applying Lemma \ref{lem:intermediate.grading} with $\tilde \tau=\tau_i+k$ and \eqref{eq:tau.formula} that, for all $\tau\geq \tau_i+1$, there are $\tilde p_1 (\tau)$ and $\tilde p_2 (\tau)$ on the $P_1$ and $P_2$ components of $M_\tau \cap B_2\backslash B_1$ respectively such that
	\begin{equation}
	| \theta(\tilde p_1(\tau))-\tilde \theta (\tilde p_2(\tau))| \geq \eta \| \theta - \underline \theta_{\tau_i+k+1} \|_{\tau_i+k+1}.
	\end{equation}
	Furthermore, using Lemma \ref{lemma:iterative.slow.decay}  and noting that $k+1=\tau+1-\tau_i -\tau'$ by \eqref{eq:tau.formula} gives
	\begin{align}
		| \theta(\tilde p_1(\tau))-\tilde \theta (p_2(\tau))| & \geq \eta e^{-q(k+1)} \| \theta - \underline \theta_{\tau_i} \|_{\tau_i} \nonumber\\
		& \geq e^{-q\tau} \eta e^{q(\tau_i+\tau'-1)} \| \theta - \underline \theta_{\tau_i} \|_{\tau_i} \nonumber\\
		& \gtrsim_{s} e^{-s\tau},\label{eq:final.bound}
	\end{align}
	where in \eqref{eq:final.bound} we used $q<s$. This is precisely \eqref{eq:variation inequality tau}.
	\end{proof}


	\section{Proof of Theorem \ref{thm:Main}}\label{sec:Proof}
	
	Let $s>0$ and let $t_s\in (0,T_0)$ be given by Proposition \ref{lem:Grading}. Then, consider the points $p_1(t)$, $p_2(t)$ from Proposition \ref{lem:Grading}, and the curve $\gamma_t \subset L_t \cap B_{2 \sqrt{T_0-t}}(x_0)$ providing a Type I bridge between these points from Lemma \ref{lem:Bridge}. Then, applying Proposition \ref{lem:Grading}, the fundamental theorem of calculus, and Lemma \ref{lem:Bridge}, in this order, we obtain:
	\begin{align}
		(T_0-t)^s & \lesssim_s  |\theta(p_2(t)) - \theta(p_1(t))| \nonumber\\
		& = \left|\int_{\gamma_t} d\theta \right| \nonumber\\
		& \lesssim \mathrm{Length}(\gamma_t) \sup_{L_t \cap B_{2 \sqrt{T_0-t}}(x_0)} |d \theta| \nonumber\\
		& \lesssim (T_0-t)^{1/2} \sup_{L_t \cap B_{2 \sqrt{T_0-t}}(x_0)} |d \theta|, \label{eq:FTC.ineq}
	\end{align}
	for all $t\in (t_s,T_0)$. 
	Rearranging \eqref{eq:FTC.ineq} and observing that $|H_{L_t}|=|d\theta|$ gives
	\begin{equation}
	\sup_{L_t \cap B_{2 \sqrt{T_0-t}}(x_0)} |H_{L_t}| \gtrsim_s (T_0-t)^{s-1/2},
	\end{equation}
	which yields \eqref{eq:s.blowup} as claimed. Then, for any $r< 1/2$, we can choose $0<s<1/2-r$ and thus
\begin{equation}
(T_0-t)^r \sup_{L_t \cap B_{R \sqrt{T_0-t}}(x_0)} |H_{L_t}| \geq (T_0-t)^{r+s-1/2} \to + \infty
\end{equation}
	as $t \nearrow T_0$, which is \eqref{eq:r.blowup} as desired.  This proves the main result, Theorem \ref{thm:Main}.


\begin{bibdiv}
		\begin{biblist}
			
			\bibitem[B07]{Bridgeland} T.~Bridgeland, \textit{Stability conditions on triangulated categories.} Ann.~of Math.~(2) {\bf 166} (2007), no.~2, 317--345.
			
		\bibitem[DBF05]{Douglas} M.~R.~Douglas, B.~Fiol, and C.~R\"omelsberger, \textit{Stability and BPS branes},  J.~High Energy Phys.~{\bf 2005.09} (2005), 006, 15 pp.
		
		\bibitem[J15]{Joyce} D.~Joyce, \textit{Conjectures on Bridgeland stability for Fukaya categories of Calabi--Yau manifolds, special Lagrangians, and Lagrangian mean curvature flow}, EMS Surv.~Math.~Sci.~{\bf 2} (2015), no.~1, 1--62.
		
		\bibitem[LT26]{Lee-Tsai} P.~H.~Lee and C.-J.~Tsai, \textit{Infinite-time singularities with vanishing mean curvature for Lagrangian mean curvature flow in Gibbons--Hawking spaces}, arXiv:2606.28767 (2026).
		
		\bibitem[LW19]{LiWang} H.~Li and B.~Wang, \textit{The extension problem of the mean curvature flow (I)}, Invent.~Math.~{\bf 218} (2019), no.~3, 721--777.
		
		\bibitem[LO25]{LO} J.~D.~Lotay and G.~Oliveira, \textit{Neck pinch singularities and Joyce conjectures in Lagrangian mean curvature flow with circle symmetry,} J.~Eur.~Math.~Soc.~(2025), published online first.
			
			
		\bibitem[LSS24]{LSS1}
		J.~D.~Lotay, F.~Schulze, and G.~Sz\'ekelyhidi,
		\textit{Ancient solutions and translators of Lagrangian mean curvature flow},
		Publ.~Math.~Inst.~Hautes \`Etudes Sci.~\textbf{140} (2024), 1--35.
		
		\bibitem[LSS25]{LSS2}
		J.~D.~Lotay, F.~Schulze, and G.~Sz\'ekelyhidi,
		\textit{Neck pinches along the Lagrangian mean curvature flow of surfaces},
		arXiv:2208.11054 (2022).
		
		\bibitem[MT07]{MorganTian} J.~W.~Morgan and G.~Tian, \textit{Ricci flow and the Poincaré conjecture}, vol.~3, American Mathematical Soc., (2007).
		
		\bibitem[N07]{Neves1} A.~Neves, \textit{Singularities of Lagrangian mean curvature flow: zero-Maslov class case},
		Invent.~Math.~\textbf{168} (2007), 449--484.
		
		\bibitem[N11]{Neves2}
		A.~Neves, \textit{Recent progress on singularities of Lagrangian mean curvature flow},
		in Surveys in geometric analysis and relativity, Adv.~Lect.~Math.~(ALM), vol.~20, Int.~Press, Somerville, MA, (2011), 413--438.
		
		\bibitem[P03]{Perelman} G.~Perelman, \textit{Ricci flow with surgery on three-manifolds}, arXiv preprint math/0303109 (2003).

		\bibitem[St23]{Stolarski} M.~Stolarski, \textit{Existence of mean curvature flow singularities with bounded mean curvature}, Duke Math.~J.~{\bf 172} (2023), no.~7, 1235--1292. 			
			
		\bibitem[SS26]{StSu} M.~Stolarski and W.-B.~Su, \textit{Spectral analysis for finite-time singularities of Lagrangian mean curvature flow}, arXiv:2606.21541 (2026).
			
		\bibitem[Sz26]{Szeke}
		G.~Sz\'ekelyhidi, \textit{Generic neck pinch singularities along 2D Lagrangian mean curvature flow}, arXiv:2602.15771 (2026).
		
		\bibitem[Th01]{Thomas} R.~P.~Thomas, \textit{Moment maps, monodromy and mirror manifolds}, Symplectic geometry and mirror symmetry (Seoul, 2000) (2001), 467--498.
		
		\bibitem[TY02]{Thomas-Yau} R.~P.~Thomas and S.-T.~Yau, \textit{Special Lagrangians, stable bundles and mean curvature flow}, Comm.~Anal. Geom.~{\bf 10} (2002), no.~5, 1075--1113. 
		
		\bibitem[Wo20]{Wood} A.~Wood, \textit{Singularities of Lagrangian mean curvature flow}, PhD thesis, UCL (University College London), (2020).
			
		\end{biblist}
	\end{bibdiv}
	

\end{document}